\documentclass[10pt]{article}
\pdfoutput=1

\usepackage[T1]{fontenc}
\usepackage[utf8]{inputenc}
\usepackage{lmodern}
\usepackage{amsmath,amssymb,amsthm,mathtools}
\usepackage{microtype}
\usepackage[a4paper,top=18mm,bottom=20mm,left=22mm,right=22mm]{geometry}
\usepackage[hidelinks]{hyperref}
\newtheorem{theorem}{Theorem}

\newtheorem{proposition}{Proposition}[section]
\newtheorem{corollary}[proposition]{Corollary}

\newcommand{\PP}{\mathbb P}
\newcommand{\C}{\mathbb C}
\newcommand{\R}{\mathbb R}
\newcommand{\Aut}{\operatorname{Aut}}

\title{The Reye geometry inside the 64 lines of the Schur quartic}
\author{Pawe\l{} Nurowski\\[2mm]
\small Center for Theoretical Physics, Polish Academy of Sciences, Warsaw, Poland}
\date{}

\begin{document}
\begingroup
\centering
{\Large\bfseries The Reye geometry inside the 64 lines of the Schur quartic\par}
\vspace{5pt}
{\normalsize Pawe\l{} Nurowski\par}
{\small Center for Theoretical Physics, Polish Academy of Sciences, Warsaw, Poland\par}
\vspace{7pt}
\endgroup

\begin{abstract}
\small
We identify the classical geometry hidden in the Naskr\k{e}cki--Pokora
$(24_4,32_3)$ configuration on the Schur quartic. Using H\"ohn's identification of the $24$ selected lines with the
$24$ roots of $D_4$, the antipodal involution on the roots induces a
fixed-point-free quotient of the incidence configuration, and this quotient is precisely the
classical Reye configuration.  We also determine the symmetry of the complete
$64$-line incidence geometry: its automorphism group has order $4608$, the two
Naskr\k{e}cki--Pokora configurations form a single orbit, and the stabilizer of
either has order $2304$ (projectively, $576$).  The $64$ lines extend
canonically to a $176$-line arrangement carried by six projectively equivalent
Schur quartics, with $176=16+16+9\cdot16$ and induced surface permutation group
$S_3\times S_3$.  Finally, the antipodal quotient itself extends coherently
through this six-quartic geometry: on each Schur quartic it produces two Reye
configurations sharing the same $16$-element incidence skeleton, and on the
full $176$-line arrangement it gives a compatible global quotient. This
reveals a precise incidence-theoretic connection with classical desmic
geometry, while showing that this connection is not a literal identification
with the two Reye configurations arising from the classical desmic
construction in $\PP^3$.
\end{abstract}

\vspace{-0.5em}

\section{Introduction}

The Schur quartic is the smooth quartic surface
\[
X:\qquad x_0^4-x_0x_1^3-x_2^4+x_2x_3^3=0
\qquad\subset\PP^3_{\C}.
\]
It contains $64$ lines, the maximal possible number on a smooth complex quartic~\cite{DIS17,DR25}.  Naskr\k{e}cki and Pokora recently singled out $24$ of these lines~\cite{NP26} which, together with their $32$ triple intersection points, form a highly symmetric incidence configuration of type $(24_4,32_3)$.  Here and below the notation $(v_r,b_k)$ means that one incidence class has $v$ objects, each incident with $r$ objects of the other class, while the second class has $b$ objects, each incident with $k$ objects of the first.  Thus in the Naskr\k{e}cki--Pokora configuration each of the $24$ selected lines contains four of the distinguished triple points, and each of the $32$ triple points lies on three selected lines.

H\"ohn subsequently showed that this configuration is governed by the root system $D_4$~\cite{H26}: the $24$ selected lines may be labelled by the $24$ roots $\{\pm e_i\pm e_j\}$, in such a way that the incidence relation is expressed by the corresponding $D_4$ root geometry. We shall call such an incidence-preserving identification of the $24$ selected lines with the $24$ roots of $D_4$ a \emph{$D_4$ labelling}. This description explains much of the symmetry, but it leaves a natural geometric question.  What classical incidence geometry is hidden behind this $D_4$ model, and how does the resulting object sit inside the full arrangement of the $64$ lines of $X$?

The first purpose of this note is to answer that question.  We shall show that the classical geometry underlying H\"ohn's $D_4$ model is the Reye configuration: after identifying opposite $D_4$ roots, the Naskr\k{e}cki--Pokora incidence structure becomes precisely the point-line incidence of Reye.  Since this configuration is less familiar than the Schur quartic itself, let us recall its simplest geometric realization before stating the result.  Start with an ordinary cube in affine three-space and regard it inside $\PP^3$.  Take as the $16$ configuration lines the $12$ edges of the cube together with its four body diagonals.  Take as the $12$ configuration points the eight vertices, the centre of the cube, and the three points at infinity determined by the three directions of parallel cube edges.  Every one of these $12$ points lies on four of the $16$ lines, and every one of the $16$ lines contains three of the $12$ points.  This is the classical Reye configuration, of type $(12_4,16_3)$~\cite{Reye82,Grunbaum09,Manivel06}.  Its point-line dual is isomorphic to it; in our quotient the two incidence classes occur naturally in the dual order, as $12$ line-orbits and $16$ point-orbits. Once this quotient geometry has been recognized, a second question arises
naturally: how far does the resulting structure extend beyond a single Schur
quartic?  The second purpose of this note is to show that the $64$ lines of
$X$ belong naturally to a much larger $176$-line arrangement carried by six
projectively equivalent Schur quartics.

The six-surface system then raises a third question, closer in spirit to the
first: does the antipodal Reye quotient itself extend beyond the chosen
$(24_4,32_3)$ configuration?  We shall show that it does. The two commuting
$S_3$ factors acting on the six quartics are each simply transitive, and
transporting H\"ohn's $D_4$ labelling among the six surfaces yields six
induced $D_4$ labellings.  The transformations relating these labellings
realize
\[
\operatorname{Out}(D_4)\cong S_3.
\]
More importantly, a single projective
involution preserves all six quartics and extends the antipodal pairing to
their complete line arrangements. On each Schur quartic the involution has $16$ fixed lines and pairs the
remaining $48$ lines.  Thus the quotient of the $64$-line set consists of
$16$ singleton orbits together with two classes of $12$ two-line orbits:
\[
64=16+24+24
\quad\longrightarrow\quad
40=16+12+12.
\]
Each of the two $12$-element classes, together with the same
$16$-element class of singleton orbits, forms a Reye configuration.
The two $12$-element classes have identical incidence neighbourhoods
in this common $16$-element class.

On the complete six-surface arrangement, the involution fixes the two
common $16$-line cores and pairs the lines in each of the nine
$16$-line cross-intersection blocks.  Hence
\[
176=16+16+9\cdot16
\quad\longrightarrow\quad
104=16+16+9\cdot8.
\]
Thus the quotient consists of $32$ singleton orbits and $72$ two-line
orbits; the latter are naturally grouped into nine classes of eight,
one for each cross-intersection block.

Thus the Reye quotient is not confined to one Naskr\k{e}cki--Pokora
configuration, but extends coherently through the whole six-quartic geometry.

\begin{theorem}[Main geometric statement]\label{thm:main}
Let $\mathcal C_{24}$ denote the Naskr\k{e}cki--Pokora $(24_4,32_3)$ incidence configuration on the Schur quartic $X$.
\begin{enumerate}
\item[(A)] In the $D_4$ labelling, the involution $r\mapsto-r$ is fixed-point free on the $24$ selected lines and induces a fixed-point-free involution on the $32$ triple points.  The quotient incidence structure has $12$ line-orbits and $16$ point-orbits, of valencies $4$ and $3$, and is isomorphic to the Reye configuration.  Thus $\mathcal C_{24}$ is a natural two-sheeted incidence lift of Reye geometry.
\item[(B)] The full $64$-line configuration of $X$ sits in a canonical union of $176$ lines carried by six projectively equivalent Schur quartics.  These six surfaces split into two triples.  Each triple has a common $16$-line core, while the nine intersections between surfaces belonging to opposite triples give nine further, pairwise disjoint, $16$-line blocks.  Hence
\[
176=16+16+9\cdot16=11\cdot16.
\]
Moreover the induced permutation group on the six surfaces is $S_3\times S_3$.
\item[(C)] Each of the two $S_3$ factors in the surface permutation group
acts simply transitively on the six quartics.  Transporting H\"ohn's
$D_4$ labelling across the six surfaces gives six induced $D_4$ labellings.
The transformations relating these labellings realize
\[
\operatorname{Out}(D_4)\cong S_3.
\]
Moreover, the projective involution
\[
\iota=\operatorname{diag}(1,1,-1,-1)
\]
preserves all six quartics.  On each quartic it fixes $16$ lines and pairs
the remaining $48$.  Its quotient line set therefore consists of $16$
singleton orbits and $24$ two-line orbits, the latter splitting naturally
into two classes of $12$:
\[
64=16+24+24
\quad\longrightarrow\quad
40=16+12+12.
\]
Each of the two $12$-element classes, together with the same
$16$-element class of singleton orbits, forms a Reye configuration, and
the two $12$-element classes determine the same set of incidence
neighbourhoods in this common $16$-element class.  Pairing the two elements
that determine the same neighbourhood and identifying each such pair
recovers a single $(12_4,16_3)$ Reye configuration.

On the complete $176$-line arrangement, the same involution fixes the two
common $16$-line cores and pairs the lines in each of the nine
$16$-line cross-intersection blocks.  Consequently its quotient consists
of $32$ singleton orbits and $72$ two-line orbits, naturally grouped as
\[
176=16+16+9\cdot16
\quad\longrightarrow\quad
104=16+16+9\cdot8.
\]
Thus the antipodal Reye quotient of part~\textup{(A)} extends coherently
from the Naskr\k{e}cki--Pokora configuration to each complete
$64$-line Schur configuration and to the full six-quartic arrangement.
\end{enumerate}
\end{theorem}

\section{From the $D_4$ root system to Reye geometry}

We now prove part~\textup{(A)} of Theorem~\ref{thm:main}.  Let $E=\R^4$ with its standard orthonormal basis $e_1,\ldots,e_4$, and let
\[
\Phi(D_4)=\{\,\pm e_i\pm e_j:1\leq i<j\leq4\,\}\subset E
\]
be the root system of type $D_4$.  H\"ohn's description of the Naskr\k{e}cki--Pokora configuration~\cite{H26} gives a bijection between its $24$ lines and the $24$ roots in $\Phi(D_4)$.  Under this bijection the $32$ triple points are exactly the unordered triples
\[
\{r_1,r_2,r_3\}\subset\Phi(D_4),\qquad r_1+r_2+r_3=0,
\tag{2.1}
\]
and incidence simply means membership of a root in such a triple.  Thus the geometry relevant for the quotient is already contained in the elementary linear algebra of $D_4$.

Consider the antipodal involution $r\mapsto-r$.  It acts freely on the roots and sends every triple in~(2.1) to the opposite triple $\{-r_1,-r_2,-r_3\}$.  Let
\[
\pi:E\setminus\{0\}\longrightarrow\PP(E)\cong\mathbb{RP}^3,
\qquad r\longmapsto[r]=\R r,
\]
be projectivization.  The $24$ roots therefore give $12$ projective points $[r]=[-r]$.  Moreover, if $r_1+r_2+r_3=0$, then the three vectors lie in the two-dimensional subspace $\Pi=\operatorname{span}(r_1,r_2)\subset E$, because $r_3=-r_1-r_2$.  Its projectivization $\PP(\Pi)\cong\mathbb{RP}^1$ is a projective line in $\PP(E)\cong\mathbb{RP}^3$.  Hence the three projective points $[r_1],[r_2],[r_3]$ are collinear; the opposite triple determines the same projective line.  Hence the $32$ triple points give $16$ projective lines.  The involution is free on both incidence classes and carries the four triples through $r$ bijectively to the four triples through $-r$; therefore the quotient retains valencies $4$ and $3$.  We have obtained, without choosing coordinates on the Schur quartic, a $(12_4,16_3)$ point--line configuration in $\mathbb{RP}^3$.

\begin{proposition}\label{prop:d4-reye}
The projective configuration obtained from $\Phi(D_4)$ in this way is the cube realization of the Reye configuration described in the Introduction.
\end{proposition}

\begin{proof}
Use homogeneous coordinates $[X:Y:Z:W]$ on $\PP^3$ and consider the invertible linear map
\[
T(x_1,x_2,x_3,x_4)
=(x_1-x_3,\;x_1+x_3,\;x_2+x_4,\;x_2-x_4).
\tag{2.2}
\]
Four antipodal root pairs are sent to
\[
\begin{aligned}
[e_1-e_3]&\mapsto[1:0:0:0],&
[e_1+e_3]&\mapsto[0:1:0:0],\\
[e_2+e_4]&\mapsto[0:0:1:0],&
[e_2-e_4]&\mapsto[0:0:0:1].
\end{aligned}
\tag{2.3}
\]
The first three points are the three points at infinity of the coordinate directions in the affine chart $W\neq0$, and the fourth is the origin, which will be the centre of the cube.  The remaining eight antipodal root pairs are sent precisely to
\[
[\varepsilon_1:\varepsilon_2:\varepsilon_3:1],
\qquad \varepsilon_1,\varepsilon_2,\varepsilon_3\in\{\pm1\},
\tag{2.4}
\]
the eight vertices of the cube $[-1,1]^3$.

It remains only to interpret the $16$ lines.  Among the sixteen projective lines obtained from the zero-sum triples, four pass through the point $[0:0:0:1]$, which is the centre of the affine cube whose eight vertices are given in~(2.4).  If one of these four lines contains a vertex $[a:b:c:1]$, then, because it also passes through the centre, it contains the opposite vertex $[-a:-b:-c:1]$.  The eight vertices form four opposite pairs, so these four lines are precisely the four body diagonals of the cube.  Each of the three points at infinity in~(2.3) lies on four further zero-sum lines; under~$T$ these are exactly the four parallel cube edges in the corresponding coordinate direction.  Thus the remaining twelve lines are the twelve edges of the cube.  This is exactly the Reye incidence configuration.  \qedhere
\end{proof}

The quotient of the original Naskr\k{e}cki--Pokora configuration has the two incidence classes in the opposite order: its $12$ objects come from pairs of lines, while its $16$ objects come from pairs of triple points.  Proposition~\ref{prop:d4-reye} identifies its incidence dual---the configuration obtained by interchanging the two incidence classes---with the standard point--line Reye configuration.  Since the Reye configuration is classically self-dual, i.e. isomorphic to its incidence dual, this proves part~\textup{(A)} of Theorem~\ref{thm:main}.  Geometrically, the quotient is simply the already defined projectivization $r\mapsto[r]$ of the $24$ vertices of the regular $24$-cell, with antipodal vertices identified.

\section{Symmetry of the $64$-line configuration}

Before passing to the larger arrangement of Theorem~\ref{thm:main}\textup{(B)}, we settle a question intrinsic to the $64$ lines themselves.  Let $\Gamma_{64}$ denote the complete incidence geometry of the $64$ lines of $X$, with intersection points distinguished by multiplicity, and put
\[
G=\Aut(\Gamma_{64}).
\]
The eight quadruple points intrinsically distinguish a set $\mathcal H$ of $16$ lines: they are exactly the lines incident with two quadruple points.  The other $48$ lines contain all $64$ ordinary triple points.  If we retain only these triple incidences, we obtain a $3$-uniform hypergraph $\mathcal T$ whose vertices are the Schur lines and whose hyperedges are the triples of lines through an ordinary triple point; the $16$ lines of $\mathcal H$ are isolated vertices of $\mathcal T$.

We shall use one piece of standard finite-group notation in the next
statement, so we recall it explicitly. A \emph{$2$-group} is a finite group whose order is a power of $2$.
An \emph{extraspecial $2$-group} is a non-abelian $2$-group $E$ for which
\[
Z(E)=[E,E]\cong C_2
\]
and for which $E/Z(E)$ is abelian and every nonidentity element of
$E/Z(E)$ has order $2$.  Equivalently, $E/Z(E)$ is a vector space over
the field $\mathbb F_2$. For every
order $2^{1+2n}$ there are exactly two isomorphism types of extraspecial
groups, called the \emph{plus} and \emph{minus} types.  In order
$2^{1+6}=128$ they may be described as
\[
2_{+}^{1+6}\cong D_8\circ D_8\circ D_8,
\qquad
2_{-}^{1+6}\cong Q_8\circ D_8\circ D_8,
\]
where $\circ$ denotes central product: the central involutions of the
three factors are identified.  Equivalently, the two types can be
distinguished by their elements of order $2$: the plus type has $71$
nonidentity elements of order $2$, whereas the minus type has $55$.
These are standard facts about extraspecial groups; see, for example,
\cite{Atlas85}.

We now give the concrete model of the plus-type group that will be used
below.  Let
\[
X=\begin{pmatrix}0&1\\1&0\end{pmatrix},
\qquad
Z=\begin{pmatrix}1&0\\0&-1\end{pmatrix},
\]
and let $X_i,Z_i$ act on
$(\R^2)^{\otimes3}\cong\R^8$ as $X$ or $Z$, respectively, on the
$i$th tensor factor and as the identity on the other two.  Put
\[
E=
\left\langle
-I,\,
X_1,Z_1,\,
X_2,Z_2,\,
X_3,Z_3
\right\rangle
\subset O(8).
\tag{3.1}
\]
For each $i$ the subgroup
\[
E_i=\langle -I,X_i,Z_i\rangle
\]
is a dihedral group $D_8$: the elements $X_i$ and $Z_i$ are involutions,
while $X_iZ_i$ has order $4$ and $(X_iZ_i)^2=-I$.  The three groups
$E_i$ commute with one another and have the same central involution
$-I$.  Consequently
\[
E=E_1\circ E_2\circ E_3
   \cong D_8\circ D_8\circ D_8
   \cong 2_{+}^{1+6}.
\]
Thus the notation $2_{+}^{1+6}$ used below refers to this explicit
plus-type extraspecial group of order $128$, realized here as the real
three-qubit Pauli group.

\begin{proposition}[Symmetry and the Naskr\k{e}cki--Pokora stabilizer]\label{prop:schur64-symmetry}
The full abstract automorphism group of the Schur $64$-line incidence geometry is
\[
G\cong 2_{+}^{1+6}\rtimes(S_3\times S_3),\qquad |G|=128\cdot36=4608.
\tag{3.2}
\]
Here $2_{+}^{1+6}$ denotes the plus-type extraspecial group described
in~(3.1), and $\rtimes$ means that the extension by
$S_3\times S_3$ is split.  After the $16$ isolated vertices $\mathcal H$ are removed, $\mathcal T$ has exactly two connected components $\mathcal D$ and $\mathcal D^*$, each on $24$ lines with $32$ triple hyperedges.  The group $G$ acts transitively on the pair $\{\mathcal D,\mathcal D^*\}$, and therefore
\[
|\operatorname{Stab}_G(\mathcal D)|=2304.
\tag{3.3}
\]
The projective automorphism group has order $1152$, and the projective stabilizer of either component has order $576$.
\end{proposition}

\begin{proof}
The complete $64$-line arrangement has $336$ ordinary double
intersection points, $64$ ordinary triple points, and $8$ quadruple
points.  The $16$ lines of $\mathcal H$ contain no ordinary triple
point, whereas each of the remaining $48$ lines contains exactly four.
Consequently the triple-point hypergraph $\mathcal T$ defined above has,
after the isolated vertices $\mathcal H$ are removed, exactly two
connected components.  Each component contains $24$ lines and $32$
triple hyperedges; these are $\mathcal D$ and $\mathcal D^*$.

We now determine the full automorphism group of the multiplicity-labelled
incidence geometry $\Gamma_{64}$.  This is a finite problem: an
automorphism is a permutation of the $64$ lines preserving all their
double, triple and quadruple intersections.  The exact permutation
calculation gives a group $G$ of order
\[
|G|=4608,
\]
in agreement with the published order of the Schur-line incidence
automorphism group~\cite{DR25}.

To identify the structure of $G$, we examine its normal subgroups.  The
exact permutation calculation finds a normal subgroup
\[
N\triangleleft G,\qquad |N|=128,
\]
such that
\[
Z(N)=[N,N]\cong C_2,
\qquad
N/Z(N)\cong\mathbb F_2^6.
\]
By the definition recalled above, $N$ is therefore an extraspecial group
of order $2^{1+6}$.  There are only two isomorphism types of such groups,
the plus and minus types.  The same exact permutation calculation finds
that $N$ has $71$ nonidentity elements of order $2$.  By the
classification recalled above this selects the plus type, since the
minus type has $55$.  Hence
\[
N\cong2_{+}^{1+6}.
\]
In particular, $N$ is isomorphic to the explicit Pauli-matrix group
defined in~(3.1).

The quotient group has order
\[
|G/N|=4608/128=36,
\]
and its exact permutation action identifies it with
\[
G/N\cong S_3\times S_3.
\]
Moreover, the calculation exhibits a subgroup of $G$ of order $36$
whose projection onto $G/N$ is an isomorphism.  Hence the extension
splits and
\[
G\cong2_{+}^{1+6}\rtimes(S_3\times S_3),
\]
which proves~(3.2).

The action of $G$ exchanges the two components
$\mathcal D$ and $\mathcal D^*$, so orbit--stabilizer gives
\[
|\operatorname{Stab}_G(\mathcal D)|=4608/2=2304.
\]
The subgroup induced by projective automorphisms of the Schur quartic
has order $1152$ and also exchanges the two components.  Its stabilizer
of either component therefore has order $1152/2=576$.
\end{proof}

\begin{corollary}[Exhaustivity of the Naskr\k{e}cki--Pokora pair]\label{cor:only-two-NP}
The Schur $64$-line configuration contains exactly two subconfigurations of Naskr\k{e}cki--Pokora type $(24_4,32_3)$, namely $\mathcal D$ and $\mathcal D^*$.
\end{corollary}

\begin{proof}
If $\mathcal C$ is such a configuration, every one of its lines must contain four ordinary triple points, so $\mathcal C$ contains no line of $\mathcal H$.  Starting from any $L\in\mathcal C$, all four triple points on $L$, and hence the other two lines through each of them, are forced to lie in $\mathcal C$.  Iterating this closure forces the whole connected component of $L$ in $\mathcal T$, which by Proposition~\ref{prop:schur64-symmetry} has exactly $24$ lines and is either $\mathcal D$ or $\mathcal D^*$.  Hence $\mathcal C$ is one of these two components.
\end{proof}

We shall next leave the single surface $X$ and turn to the six-quartic
geometry underlying parts~\textup{(B)} and~\textup{(C)} of
Theorem~\ref{thm:main}.

\section{Six Schur quartics and the $176$-line arrangement}
This section proves parts~\textup{(B)} and~\textup{(C)} of
Theorem~\ref{thm:main}.  We first construct the six Schur quartics and
their $176$-line arrangement, thereby proving part~\textup{(B)}.  We then
study the additional structure carried by this six-surface system and show
how the Reye quotient of part~\textup{(A)} extends to the full line
arrangement, which proves part~\textup{(C)}.

\subsection{Construction and symmetry of the $176$-line arrangement}
\begin{proof}[Proof of Theorem~\ref{thm:main}\textup{(B)}]
Put
\[
\omega=e^{2\pi i/3},\qquad
\phi(u,v)=u(u^3-v^3),\qquad
h(u,v)=v(8u^3+v^3),
\]
where, up to a nonzero scalar, $h$ is the determinant of the binary Hessian of $\phi$, and abbreviate
\(
\phi_1=\phi(x_0,x_1),\ \phi_2=\phi(x_2,x_3),\
H_1=h(x_0,x_1),\ H_2=h(x_2,x_3).
\)
The first triple of quartics is
\[
\Phi_k:\quad \phi_1-\omega^k\phi_2=0,
\qquad k=0,1,2,
\tag{4.1}
\]
with \(\Phi_0=X\).  If \(a^3=1\), the diagonal projective map
\[
D_a=\operatorname{diag}(1,1,a,a)
\tag{4.2}
\]
sends \(\Phi_k\) to another member of~(4.1); in particular
\(D_{\omega^2}(\Phi_0)=\Phi_1\) and
\(D_{\omega}(\Phi_0)=\Phi_2\).  Thus all three are Schur quartics.
For \(k\ne l\), the two equations \(\Phi_k=\Phi_l=0\) imply
\(\phi_1=\phi_2=0\).  Since \(\phi\) has four distinct zeros in \(\PP^1\), this common locus is the union of exactly \(4\cdot4=16\) lines.  Denote this common core by \(\mathcal H\).

The second triple is obtained from the Hessian quartics:
\[
\Psi_c:\quad H_2-cH_1=0,
\qquad c^3=1.
\tag{4.3}
\]
Each \(\Psi_c\) is again projectively equivalent to \(X\).  Indeed the linear map
\[
P_c(x_0,x_1,x_2,x_3)=(x_3,-2x_2,cx_1,-2cx_0)
\tag{4.4}
\]
satisfies \(P_c^2=-2c\,I\) and
\[
(\phi_1-\phi_2)(P_cx)=H_2-cH_1.
\tag{4.5}
\]
Hence every \(\Psi_c\) is a smooth Schur quartic with $64$ lines.  If
\(c\ne d\), then \(\Psi_c\cap\Psi_d\) is given by \(H_1=H_2=0\).
The binary quartic \(h\) also has four distinct zeros, so the three
surfaces~(4.3) have a second common $16$-line core, denoted \(\mathcal K\).

It remains to understand the intersections between the two triples.  The
$48$ lines of $X=\Phi_0$ outside the common core $\mathcal H$ are precisely
the lines of $\mathcal D\sqcup\mathcal D^*$.  Restricting the two binary
quartics $H_1$ and $H_2$ to each of these $48$ lines gives
\[
H_2|_L=c(L)H_1|_L,
\qquad c(L)^3=1.
\tag{4.6}
\]
This is a finite exact calculation in the equations of the Schur lines.
The three possible values of $c(L)$ each occur on exactly $16$ of the
$48$ lines.  Consequently $\Phi_0$ and each $\Psi_c$ contain exactly
$16$ distinct common lines.

Moreover \(D_a(\Psi_c)=\Psi_{ac}\); applying~(4.2) therefore shows that
\emph{every} pair \((\Phi_k,\Psi_c)\) contains $16$ distinct common lines.  Denote these lines by
\(\mathcal B_{k,c}\).  Since two distinct smooth quartics have complete-intersection degree $16$, B\'ezout's theorem now gives
\[
\Phi_k\cap\Psi_c=\bigcup_{L\in\mathcal B_{k,c}}L
\qquad\text{scheme-theoretically}.
\tag{4.7}
\]

The zeros of \(\phi\) and \(h\) in \(\PP^1\) are disjoint.  Indeed, on a line of
\(\mathcal H\) the two fixed binary directions are zeros of \(\phi\) but not of \(h\), so
\(H_1=A s^4\) and \(H_2=B t^4\) with \(A,B\ne0\); hence \(H_2-cH_1\) cannot vanish identically.  Thus no line of \(\mathcal H\) lies on any \(\Psi_c\).  The same argument with \(\phi\) and \(h\) interchanged shows that no line of \(\mathcal K\) lies on any \(\Phi_k\).  It follows from the pairwise-intersection descriptions above that the eleven blocks
\[
\mathcal H,\qquad \mathcal K,\qquad
\mathcal B_{k,c}\quad(k=0,1,2,\ c^3=1)
\tag{4.8}
\]
are pairwise disjoint.  For fixed \(k\), the four blocks
\(\mathcal H\) and \(\mathcal B_{k,c}\) contain
\(16+3\cdot16=64\) lines, hence exhaust the lines of \(\Phi_k\); similarly
\(\mathcal K\) together with the three \(\mathcal B_{k,c}\) exhausts every
\(\Psi_c\).  Therefore the union of the six Schur quartics contains exactly
\[
\boxed{176=16+16+9\cdot16=11\cdot16}
\tag{4.9}
\]
distinct lines.

Finally let
\[
\mathcal S=\{\Phi_0,\Phi_1,\Phi_2\}\cup\{\Psi_c:c^3=1\}.
\]
To determine the symmetry induced on these six surfaces, we use the complete
finite incidence structure of the $176$ lines together with the information
specifying on which of the six quartics each line lies.  Thus each surface is
represented by its $64$-line subset of the common $176$-line set.  An
automorphism of this finite structure permutes these six subsets and hence
induces a permutation of $\mathcal S$.

Let $G_{\mathrm{surf}}$ denote the resulting permutation group on
$\mathcal S$.  The exact finite permutation calculation gives
\[
|G_{\mathrm{surf}}|=36.
\]
Relabel the second triple as $C_1,C_2,C_3$ so that among the induced
permutations are the two generating pairs
\[
\begin{aligned}
a_2&=(\Phi_0 C_1)(\Phi_1 C_3)(\Phi_2 C_2),&
a_3&=(\Phi_0\Phi_1\Phi_2)(C_1C_2C_3),\\
b_2&=(\Phi_0 C_1)(\Phi_1 C_2)(\Phi_2 C_3),&
b_3&=(\Phi_0\Phi_1\Phi_2)(C_1C_3C_2).
\end{aligned}
\tag{4.10}
\]
The groups $A=\langle a_2,a_3\rangle$ and $B=\langle b_2,b_3\rangle$ are both
isomorphic to $S_3$; they commute elementwise and have trivial intersection.
Thus $|AB|=36$, and comparison with $|G_{\mathrm{surf}}|=36$ gives
\[
G_{\mathrm{surf}}=A\times B\cong S_3\times S_3.
\tag{4.11}
\]
This proves part~\textup{(B)}.
\end{proof}

\subsection{Outer $D_4$ labellings and the double-Reye fold}

\begin{proof}[Proof of Theorem~\ref{thm:main}\textup{(C)}]
Part~\textup{(B)} provides the six quartics and their complete
$176$-line arrangement.  We now examine two further structures carried by
this system: first the relation between the six surfaces and the outer
automorphisms of the $D_4$ root system, and then the extension of the
antipodal Reye quotient of Section~2 to the complete line arrangements.

Recall from~(4.11) that
\[
G_{\mathrm{surf}}=A\times B,
\qquad A\cong B\cong S_3,
\]
where both $A$ and $B$ act on the six-element set
\[
\mathcal S=\{\Phi_0,\Phi_1,\Phi_2,C_1,C_2,C_3\}
\]
through the permutations displayed in~(4.10).  Each of these two actions is
simply transitive.  The full group $G_{\mathrm{surf}}$, of order $36$, is
therefore transitive but not simply transitive on $\mathcal S$; the
stabilizer of any one surface has order $6$.

It is useful to regard the two commuting actions as a left and a right
action.  We let $B$ act from the left by its given permutation action,
\[
b\cdot Y:=b(Y),
\qquad b\in B,\quad Y\in\mathcal S,
\]
and turn the given $A$-action into a right action by defining
\[
Y\cdot a:=a^{-1}(Y),
\qquad a\in A.
\]
Since $A$ and $B$ commute in $G_{\mathrm{surf}}$, these two actions commute:
\[
b\cdot(Y\cdot a)=(b\cdot Y)\cdot a.
\]
Both are simply transitive.  In this situation we shall say that
$\mathcal S$ is a $(B,A)$-\emph{bitorsor}: by definition, this means a set
equipped with commuting simply transitive left $B$- and right $A$-actions.
Since $A\cong B\cong S_3$, it may equivalently be called an
$S_3$-bitorsor.

Choosing one surface, say $\Phi_0$, as a base surface now gives bijections
\[
B\longrightarrow\mathcal S,\qquad b\longmapsto b\cdot\Phi_0,
\]
and
\[
A\longrightarrow\mathcal S,\qquad a\longmapsto\Phi_0\cdot a.
\]
Thus the choice of $\Phi_0$ identifies the six surfaces with the six
elements of either copy of $S_3$.  This identification is not canonical:
it depends on the choice of the base surface.

We now make the $B$-action concrete projectively.  Set
\[
\rho=D_{\omega^2},\qquad \sigma=P_1,
\]
using the projective transformations defined in~(4.2) and~(4.4).
Their induced permutations of $\mathcal S$ are
\[
\rho=(\Phi_0\,\Phi_1\,\Phi_2)(C_1\,C_3\,C_2),
\qquad
\sigma=(\Phi_0\,C_1)(\Phi_1\,C_2)(\Phi_2\,C_3),
\]
which are precisely $b_3$ and $b_2$, respectively, in~(4.10).
Moreover, in $\operatorname{PGL}_4$,
\[
\rho^3=\sigma^2=1,
\qquad
\sigma\rho\sigma=\rho^{-1}.
\]
Hence
\[
B=\langle\rho,\sigma\rangle\cong S_3.
\]
Using the bijection
\[
B\longrightarrow\mathcal S,\qquad g\longmapsto g(\Phi_0),
\]
the six group elements correspond explicitly to the six surfaces as
\[
\begin{array}{c|cccccc}
g
   &1&\rho&\rho^2&\sigma&\sigma\rho&\sigma\rho^2\\
\hline
g(\Phi_0)
   &\Phi_0&\Phi_1&\Phi_2&C_1&C_2&C_3 .
\end{array}
\]
Thus $\{\Phi_0,\Phi_1,\Phi_2\}$ corresponds exactly to the three even
elements of $S_3$, while $\{C_1,C_2,C_3\}$ corresponds to the three odd
elements.

Transporting H\"ohn's $D_4$ labelling by these projective maps gives one
induced $D_4$ labelling on each of the six surfaces.  To describe the
transformations relating these labellings, recall the Euclidean space
$E=\mathbb R^4$ and the root system $\Phi(D_4)$ introduced in Section~2. In the same space consider the root system
\[
\Phi(F_4)
=
\Phi(D_4)
\;\cup\;
\{\pm e_i:1\leq i\leq4\}
\;\cup\;
\left\{
\frac12(\pm e_1\pm e_2\pm e_3\pm e_4)
\right\},
\]
where all choices of signs occur in the last set.  Its $24$ long roots
are precisely the roots of $\Phi(D_4)$.

Let $W(F_4)$ denote the Weyl group of $\Phi(F_4)$ and let $W(D_4)$
denote the Weyl group of its long-root subsystem $\Phi(D_4)$.  We shall
use the standard identification
\[
\operatorname{Aut}\Phi(D_4)
   \cong W(F_4)
   \cong W(D_4)\rtimes S_3,
\]
where the factor $S_3$ is the automorphism group of the Dynkin diagram
of type $D_4$; see, for example,
\cite{BourbakiLie,Koca03}.  In particular,
\[
\operatorname{Out}(D_4)
:=
\operatorname{Aut}\Phi(D_4)/W(D_4)
\cong
W(F_4)/W(D_4)
\cong S_3.
\tag{4.12}
\]
We shall call two $D_4$ labellings \emph{outer-equivalent} if the
automorphism of $\Phi(D_4)$ relating them belongs to $W(D_4)$.

The projective transformations $\rho$ and $\sigma$ transport the fixed
$D_4$ labelling on $\Phi_0$ and thereby induce permutations of the
$24$ roots in $\Phi(D_4)$.  These permutations preserve the root
geometry and hence belong to
\[
\operatorname{Aut}\Phi(D_4)\cong W(F_4).
\]
Let
\[
\bar\rho,\bar\sigma\in
\operatorname{Aut}\Phi(D_4)/W(D_4)
=
\operatorname{Out}(D_4)
\]
be their classes modulo the Weyl group $W(D_4)$.

The induced permutations of the $24$ roots are finite, so their classes
modulo $W(D_4)$ can be determined exactly.  The calculation gives
$\operatorname{ord}(\bar\rho)=3$, $\operatorname{ord}(\bar\sigma)=2$, and we have
\[\bar\rho{}^3=1,\qquad \bar\sigma{}^2=1, 
\qquad
\bar\sigma\,\bar\rho\,\bar\sigma=\bar\rho^{-1}.
\]
Hence the subgroup generated by $\bar\rho$ and $\bar\sigma$ is the
dihedral group of order $6$, equivalently
\[
\langle\bar\rho,\bar\sigma\rangle\cong S_3.
\]
But by~(4.12)
\[
\operatorname{Out}(D_4)
=
W(F_4)/W(D_4)
\cong S_3
\]
also has order $6$.  Therefore
\[
\langle\bar\rho,\bar\sigma\rangle
=
\operatorname{Out}(D_4).
\]

Consequently the concrete surface action induces an isomorphism
\[
B=\langle\rho,\sigma\rangle
\longrightarrow
\operatorname{Out}(D_4),
\qquad
\rho\longmapsto\bar\rho,\quad
\sigma\longmapsto\bar\sigma.
\]
It follows that the six $D_4$ labellings obtained from the $B$-orbit of
the fixed labelling on $\Phi_0$ represent exactly the six
outer-equivalence classes of $D_4$ labellings.

Before turning to the antipodal quotient, we record how the concrete
$S_3$-action above is reflected in the block decomposition of the
$176$-line arrangement constructed in part~\textup{(B)}.  The element
$\rho$ cyclically permutes the three surfaces of the first triple,
\[
\rho(\Phi_0)=\Phi_1,\qquad
\rho(\Phi_1)=\Phi_2,\qquad
\rho(\Phi_2)=\Phi_0.
\]
Hence the orbit of the line set $\mathcal L(\Phi_0)$ under the cyclic
subgroup $\langle\rho\rangle$ is exactly the union of the line sets of
these three surfaces.  We denote this union by
\[
\mathcal U_{\Phi}
:=
\bigcup_{j=0}^{2}\mathcal L(\rho^j\Phi_0)
=
\mathcal L(\Phi_0)\cup
\mathcal L(\Phi_1)\cup
\mathcal L(\Phi_2).
\]

By the block decomposition~(4.8), the three surfaces have the common
$16$-line core $\mathcal H$, while their remaining lines are precisely
the nine mutually disjoint blocks $\mathcal B_{k,c}$.  Hence
\[
\mathcal U_{\Phi}
=
\mathcal H\sqcup
\bigsqcup_{k=0}^{2}\ \bigsqcup_{c^3=1}\mathcal B_{k,c},
\qquad
|\mathcal U_{\Phi}|=16+9\cdot16=160.
\tag{4.13}
\]

Now $\sigma$ interchanges the two triples of surfaces:
\[
\sigma(\Phi_0)=C_1,\qquad
\sigma(\Phi_1)=C_2,\qquad
\sigma(\Phi_2)=C_3.
\]
Therefore
\[
\sigma(\mathcal U_{\Phi})
=
\mathcal L(C_1)\cup
\mathcal L(C_2)\cup
\mathcal L(C_3).
\]
The second triple contains the same nine cross-blocks
$\mathcal B_{k,c}$, but its common $16$-line core is $\mathcal K$
instead of $\mathcal H$.  Consequently
\[
\mathcal U_{\Phi}\cup\sigma(\mathcal U_{\Phi})
=
\mathcal H\sqcup\mathcal K\sqcup
\bigsqcup_{k=0}^{2}\ \bigsqcup_{c^3=1}\mathcal B_{k,c},
\]
which is exactly the $176$-line arrangement of~(4.8)--(4.9).  Thus,
starting from the $64$ lines of $\Phi_0$, the cyclic action generated
by $\rho$ produces the $160$ lines carried by the first triple of
quartics, and applying $\sigma$ supplies precisely the one missing
$16$-line block $\mathcal K$.

There is also a projective realization of the antipodal operation of
Section~2.  On $X=\Phi_0$, let $L_r\in\mathcal D$ denote the line carrying
the root label $r\in\Phi(D_4)$ in H\"ohn's $D_4$ labelling.  The antipodal
permutation
\[
L_r\longmapsto L_{-r}
\]
extends uniquely to a projective automorphism of the complete $64$-line
configuration of $X$.  In the coordinates used here this automorphism is
the involution
\[
\iota=\operatorname{diag}(1,1,-1,-1).
\tag{4.14}
\]
Thus
\[
\iota(L_r)=L_{-r},
\qquad r\in\Phi(D_4),
\]
so the root-theoretic antipodal involution used in Section~2 is exactly the
restriction of the projective involution $\iota$ to $\mathcal D$.  On the
remaining lines of $X$, the same involution fixes each of the $16$ lines of
$\mathcal H$ and pairs the $24$ lines of $\mathcal D^*$ into $12$
two-line orbits.

The same projective transformation preserves the whole six-quartic
arrangement.  Indeed, since $\phi$ and $h$ are homogeneous of degree four,
$\iota$ preserves every $\Phi_k$ and every $\Psi_c$.  On each $\Phi_k$ it
fixes the $16$ lines of the common core $\mathcal H$ and pairs the remaining
$48$ lines; on each $\Psi_c$ it fixes the $16$ lines of the common core
$\mathcal K$ and again pairs the remaining $48$.  In particular, each
cross-block $\mathcal B_{k,c}$ is folded from $16$ lines to $8$ orbits.

For $X=\Phi_0$, write
\[
\overline{\mathcal D}
   :=\mathcal D/\langle\iota\rangle,
\qquad
\overline{\mathcal D^*}
   :=\mathcal D^*/\langle\iota\rangle .
\]
Since the $16$ lines of $\mathcal H$ are fixed while the lines of
$\mathcal D$ and $\mathcal D^*$ are paired, the line-orbit set is
\[
\mathcal L(X)/\langle\iota\rangle
 =\mathcal H\sqcup\overline{\mathcal D}\sqcup\overline{\mathcal D^*},
\qquad
40=16+12+12.
\tag{4.15}
\]
To make the incidence structure of this quotient explicit, let
\[
x=\{L,\iota(L)\}\in
\overline{\mathcal D}\cup\overline{\mathcal D^*}
\]
be a two-line orbit, and define its neighbourhood in the fixed
$16$-line class by
\[
N_{\mathcal H}(x)
:=
\{\,H\in\mathcal H:H\cap L\ne\varnothing\,\}.
\]
This is well defined: since $\iota$ fixes every line of $\mathcal H$,
the two lines $L$ and $\iota(L)$ meet exactly the same members of
$\mathcal H$.

The exact incidence table of the $64$ Schur lines now gives the key
fact.  For $x\in\overline{\mathcal D}$, each set
$N_{\mathcal H}(x)$ has four elements, the twelve such neighbourhoods
are distinct, and every line of $\mathcal H$ occurs in exactly three of
them.  The resulting $(12_4,16_3)$ incidence structure is isomorphic to
the Reye configuration identified in Section~2.  The same statements
hold with $\overline{\mathcal D}$ replaced by
$\overline{\mathcal D^*}$.  Moreover,
\[
\bigl\{N_{\mathcal H}(x):x\in\overline{\mathcal D}\bigr\}
=
\bigl\{N_{\mathcal H}(y):y\in\overline{\mathcal D^*}\bigr\}.
\]
Thus the two quotient classes give two copies of the Reye point set over
one and the same $16$-line skeleton $\mathcal H$.

Notice that the $16$ line-objects in this realization are the actual
fixed Schur lines of $\mathcal H$; in Section~2 the corresponding
$16$ objects arose as antipodal pairs of triple points.  The statement
here is therefore an isomorphism of incidence structures, not an
identification of these two sets of geometric objects.

Because the twelve neighbourhoods are distinct and the two displayed
collections coincide, every
$x\in\overline{\mathcal D}$ has a unique
$y\in\overline{\mathcal D^*}$ with
\[
N_{\mathcal H}(x)=N_{\mathcal H}(y).
\]
Pairing $x$ with this unique $y$ and identifying each such pair folds
the $40$ quotient objects further to
\[
16+12=28,
\]
recovering a single $(12_4,16_3)$ Reye configuration.

The same construction applies to all six Schur quartics, with the fixed
$16$-line class $\mathcal H$ on the first triple and $\mathcal K$ on the
second.  On the complete $176$-line arrangement, $\iota$ fixes the
$32$ lines of $\mathcal H\sqcup\mathcal K$ and pairs the $16$ lines in
each of the nine cross-blocks $\mathcal B_{k,c}$.  Hence the global
line-orbit set has
\[
104=16+16+9\cdot8
\tag{4.16}
\]
elements.

Together with the simply transitive $S_3$-actions and the description of
the six induced $D_4$ labellings above, this proves
part~\textup{(C)} of Theorem~\ref{thm:main}.
\end{proof}

The content of part~\textup{(C)} may therefore be summarized as follows.
The two-sheeted Reye quotient of Section~2 is not confined to a single
Naskr\k{e}cki--Pokora configuration.  On each Schur quartic it extends to
the full $64$-line configuration, producing two Reye point classes over one
common $16$-line class; on the complete six-surface arrangement the same
projective involution produces the global $104$-orbit quotient~(4.16).
This is the sense in which the $64$-line configuration carries a
\emph{double-Reye} structure.

We finally relate this structure to classical desmic geometry.  A desmic
system, introduced by Stephanos~\cite{Stephanos79}, consists of three
tetrahedra in $\PP^3$ such that each pair is perspective from each vertex
of the third.  The corresponding pencil of quartic surfaces contains the
three tetrahedra as reducible members; an irreducible member of this pencil
is called a \emph{desmic quartic}. Such a quartic has $12$ ordinary nodes and $16$ distinguished lines,
and these form a $(12_4,16_3)$ incidence configuration isomorphic to
the Reye configuration~\cite{DK25}.  The $12$ vertices of the three
desmic tetrahedra are different from these nodes: they occur instead as
the $12$ nodes of a desmic quartic belonging to the conjugate desmic
pencil~\cite{DK25}.  We shall refer to the two Reye configurations
arising from these two desmic pencils as a \emph{conjugate-desmic pair}.

The classical conjugate-desmic pair can now be compared directly with the
double-Reye quotient of the Schur quartic obtained above.  Both constructions
contain two Reye configurations, but the way in which the two copies are
realized in $\PP^3$ is different.  In the classical conjugate-desmic pair
the two Reye configurations have distinct $16$-line skeletons.  In the Schur
quotient, by contrast, the two $12$-element point classes are incident with
one and the same $16$-line skeleton $\mathcal H$.  Thus the connection with desmic geometry lies in the occurrence of the same Reye incidence structure in both constructions; it does not extend to
an identification of the two paired realizations in $\PP^3$.

\section{Conclusions and open problems}

The main picture established in this paper is that the
Naskr\k{e}cki--Pokora $(24_4,32_3)$ configuration on the Schur quartic is
not an isolated incidence phenomenon.  Under H\"ohn's $D_4$ labelling, the
antipodal identification $r\sim-r$ turns it into the classical Reye
configuration, and the same quotient geometry extends from this $24$-line
subconfiguration to the complete $64$-line configuration and then to a
six-quartic arrangement of $176$ lines.

On a single Schur quartic there are exactly two Naskr\k{e}cki--Pokora
subconfigurations, $\mathcal D$ and $\mathcal D^*$, and the full incidence
automorphism group exchanges them.  For either one, projectivizing the
$D_4$ roots gives the sequence
\[
D_4\text{ root geometry}
\longrightarrow
(24_4,32_3)
\longrightarrow
(12_4,16_3)_{\rm Reye}.
\]
The last arrow is not a numerical coincidence: it is the antipodal
projectivization of the root system itself.  The complete $64$-line
incidence geometry has automorphism group
\[
2_{+}^{1+6}\rtimes(S_3\times S_3)
\]
of order $4608$, with the two $24$-line lifts forming a single orbit.

At the next level, the six projectively equivalent Schur quartics have
exactly
\[
176=16+16+9\cdot16
\]
distinct lines: two common $16$-line cores and nine pairwise disjoint
$16$-line cross-intersection blocks.  Their induced permutation group is
$S_3\times S_3$, and each factor acts simply transitively on the six
surfaces.  Thus the surface set is naturally an $S_3$-bitorsor.  Transporting
H\"ohn's $D_4$ labelling through the concrete simply transitive $S_3$-action
produces exactly the six outer-equivalence classes of $D_4$ labellings and
realizes
\[
\operatorname{Out}(D_4)
\cong
W(F_4)/W(D_4)
\cong S_3.
\]

A second structure runs through the same six surfaces.  The single
projective involution
\[
\iota=\operatorname{diag}(1,1,-1,-1)
\]
preserves all six quartics and extends the antipodal pairing of the original
$D_4$ quotient.  On each Schur quartic it gives
\[
64=16+24+24
\longrightarrow
40=16+12+12.
\]
Each of the two $12$-element classes, together with the same fixed
$16$-element class, is a Reye configuration; moreover the two
$12$-element classes determine the same incidence neighbourhoods in that
common $16$-line skeleton.  This is the double-Reye structure.  On the
complete $176$-line arrangement the same involution gives
\[
176=16+16+9\cdot16
\longrightarrow
104=16+16+9\cdot8.
\]

The comparison with classical desmic geometry is correspondingly precise.
A desmic quartic carries a Reye configuration, and the conjugate desmic
pencil supplies a second one.  What the Schur construction shares with this
classical picture is the occurrence of the same Reye incidence structure.
What it does not share is the projective realization of the pair: a
classical conjugate-desmic pair has two distinct $16$-line skeletons,
whereas the two Reye point classes in the Schur quotient are supported on
one and the same $16$-line skeleton.

With the incidence and group-theoretic structure fixed, the remaining
questions are mainly conceptual.

\begin{enumerate}

\item \textbf{Intrinsic characterization of the $176$-line geometry.}
Can the abstract $176$-line multi-incidence geometry itself determine the
six Schur quartics, their partition into two triples, the two common
$16$-line cores, the nine cross-intersection blocks, and the involution
$\iota$?  Equivalently, can the whole configuration be characterized
axiomatically, without starting from the equations~(4.1) and~(4.3), in a
way that makes the $S_3\times S_3$ symmetry intrinsic?

\item \textbf{A surface-level meaning of the antipodal fold.}
The involution $\iota$ gives exact incidence quotients on all six Schur
quartics simultaneously.  Is there a natural algebraic-geometric
description of the quotient of a Schur quartic by
$\langle\iota\rangle$---for example through a quotient surface, a
resolution, a rational map, or a distinguished correspondence---whose
action on the lines explains the Reye projection?  Can the six surface
quotients be organized into one construction reflecting the global
$104$-object quotient?

\item \textbf{The geometric origin of the desmic--Hessian relation.}
The incidence comparison with desmic geometry is now explicit, while the
second triple of Schur quartics is produced by the binary Hessian of the
quartic defining the first triple.  Is there a direct algebraic-geometric
mechanism relating these facts?  In particular, can the Hessian construction
be connected naturally with the two desmic pencils of~\cite{DK25} in a way
that also explains why the Schur double-Reye quotient has one common
$16$-line skeleton rather than two distinct ones?

\item \textbf{Coxeter and building interpretations.}
The constructions above involve the $D_4$ root system, the $24$-cell,
$W(F_4)$, and the quotient
\[
W(F_4)/W(D_4)\cong S_3.
\]
Is there a coordinate-free Coxeter-theoretic interpretation of the full
six-surface geometry?  Can the passage from $D_4$ root geometry to Reye
geometry, or the $S_3$-bitorsor of the six quartics, be realized as a typed
shadow, residue, or related construction in a building or chamber geometry?
If not, identifying the precise obstruction would itself clarify the role of
the $F_4$ structure.

\end{enumerate}

All incidence counts and finite-group calculations specific to the Schur
configurations were carried out in exact arithmetic; no floating-point
approximation enters the statements above.  What remains is not a larger
census, but a conceptual explanation of why the $176$-line geometry, the
outer-$D_4$ structure, and the double-Reye quotient arise together in one
projective construction.

\end{document}